\documentclass[11pt]{article}
\usepackage[margin=1.1in]{geometry}
\usepackage[utf8]{inputenc}
\usepackage{amsmath,amssymb}
\usepackage{amsthm}
\usepackage{amsfonts}
\usepackage{amssymb}
\usepackage{mathtools}
\usepackage{mathrsfs}

\usepackage{enumerate}
\usepackage{hyperref}

\newtheorem{theorem}{Theorem}

\newtheorem{lemma}{Lemma}
\newtheorem{corollary}{Corollary}

\newtheorem{observation}{Observation}

\newcounter{claim}
\newenvironment{claim}[1][]
{\refstepcounter{claim} \begin{trivlist} \item[] {\bf Claim~\theclaim.}\space#1 \itshape}
{\end{trivlist}}

\newenvironment{cpf}
{\begin{trivlist} \item[] {\em Proof of claim. }}
{$\hfill\diamond$ \end{trivlist}}

\newcommand{\transp}{\mathsf T}
\newcommand{\conv}{\operatorname{conv}}

\newcommand{\R}{\mathbb R}
\newcommand{\N}{\mathbb N}
\newcommand{\Z}{\mathbb Z}

\newcommand{\A}{\mathcal{A}}
\newcommand{\tA}{\tilde{\mathcal{A}}}

\newcommand{\order}{\preceq}
\newcommand{\georder}{\succeq}

\begin{document}

\title{Integer packing sets form a well-quasi-ordering}

\author{
Alberto Del Pia
\thanks{Department of Industrial and Systems Engineering \& Wisconsin Institute for Discovery,
             University of Wisconsin-Madison.
             E-mail: {\tt delpia@wisc.edu}.
             }
\and
Dion Gijswijt
\thanks{Delft Institute of Applied Mathematics,
             Delft University of Technology.
             E-mail: {\tt d.c.gijswijt@tudelft.nl}.
             }
\and
Jeff Linderoth
\thanks{Department of Industrial and Systems Engineering,
             University of Wisconsin-Madison.
             E-mail: {\tt linderoth@wisc.edu}.
             }
\and
Haoran Zhu
\thanks{Department of Industrial and Systems Engineering,
             University of Wisconsin-Madison.
             E-mail: {\tt hzhu94@wisc.edu}.
             }
}

\maketitle

\begin{abstract}
An integer packing set is a set of non-negative integer vectors with the property that, if a vector $x$ is in the set, then every non-negative integer vector $y$ with $y \le x$ is in the set as well.
Integer packing sets appear naturally in Integer Optimization.
In fact, the set of integer points in any packing polyhedron is an integer packing set.
The main result of this paper is that integer packing sets, ordered by inclusion, form a well-quasi-ordering.

This result allows us to answer a question recently posed by Bodur et al.
In fact, we prove that the $k$-aggregation closure of any packing polyhedron is again a packing polyhedron.
The generality of our main result allows us to provide a generalization to non-polyhedral sets: 
The $k$-aggregation closure of any downset of $\R^n_+$ is a packing polyhedron.
\end{abstract}


\section{Introduction}
\label{sec intro}
In order theory, a \emph{quasi-order} is a binary relation $\order$ over a set $X$ that is \emph{reflexive:} $\forall a \in X$, $a \order a$, and \emph{transitive: } $\forall a,b,c \in X$, $a \order b$ and $b \order c$ imply $a \order c$. A quasi-order $\order$ is a \emph{well-quasi-order} (\emph{wqo}) if for any infinite sequence $x_1, x_2, \ldots$ of elements from $X$ there are indices $i<j$ such that $x_i \order x_j$.

A classic example of a quasi-order over the set of graphs is given by the graph minor relation. The  Robertson-Seymour Theorem (also known as the graph minor theorem) essentially states that the set of finite graphs is well-quasi-ordered by the graph minor relation. This fundamental result is the culmination of twenty papers written as part of the Graph Minors Project~\cite{robertson2004graph}. Interested readers may find more examples and characterizations in the comprehensive survey paper by Kruskal~\cite{kruskal1972theory}. The main result of this paper is that a quasi-order arising from Integer Optimization is a well-quasi-order.

Let $\N=\{0,1,2,\ldots\}$ denote the set of nonnegative integers and let $[n]=\{1,2,\ldots, n\}$ for any $n\in \N$. We define an \emph{integer packing set} in $\R^n$ as a subset $Q$ of $\N^n$ with the property that: if $x \in Q$, $y \in \N^n$ and $y \le x$, then $y \in Q$.
Note that the relation $\subseteq$ is a quasi-order over the set of integer packing sets.
We are now ready to state our main result.
\begin{theorem}
\label{theo: packingset wqo}
The set of integer packing sets in $\R^n$ is well-quasi-ordered by the relation $\subseteq$.
\end{theorem}

Integer packing sets appear naturally in Integer Optimization.
A \emph{packing polyhedron} is a set of the form $P = \{x \in \R^n \mid A x \leq b, \ x \geq 0\}$ where the data $A \in \R^{n \times m}, b \in \R^m$ is non-negative.
Clearly, for any packing polyhedron $P$, the set $P \cap \Z^n$, is an integer packing set. However, note that not all integer packing sets are of this form.
This connection between packing polyhedra and integer packing sets allows us to employ Theorem~\ref{theo: packingset wqo} to answer a recently posed open question in Integer Optimization.

In \cite{bodur2017aggregation}, the authors introduce the concept of $k$-aggregation closure for packing and covering polyhedra. Given a packing polyhedron $P = \{x \in \R^n \mid A x \leq b, \ x \geq 0\}$, and a positive integer $k$, the \emph{$k$-aggregation closure} of $P$ is defined by
\begin{equation*}
\A_k(P): = \bigcap_{\lambda^1, \ldots, \lambda^k \in \mathbb{R}^m_+} \conv(\{  x \in \N^n \mid (\lambda^j)^\transp A x \leq (\lambda^j)^\transp b, \ \forall j\in[k] \}).
\end{equation*}

The set $\A_k(P)$ is defined as the intersection of an infinite number of sets, each of which is the convex hull of an integer packing set. A natural question, posed in \cite{bodur2017aggregation}, is whether the set $\A_k(P)$ is polyhedral.
The authors provide a partial answer to this question by showing that $\A_k(P)$ is a polyhedron, provided that every entry of $A$ is positive.
As a consequence of Theorem~\ref{theo: packingset wqo}, we give a complete answer to the posed question.
\begin{theorem}
\label{theo: k-aggregation}
For any packing polyhedron $P$ and any $k \ge 1$, the set $\A_k(P)$ is a packing polyhedron.
\end{theorem}

The generality of our proving techniques allows us to provide a generalization of Theorem~\ref{theo: k-aggregation} to the setting where the given set is a downset of $\R^n_+$ instead of a polyhedron.
We recall that a \emph{downset} of $\R^n_+$ is a subset $D$ of $\R^n_+$ with the property that, if $x \in D$, $y \in \R^n_+$ and $y \leq x$, then $y \in D$.
Clearly, a packing polyhedron in $\R^n$ is a downset of $\R^n_+$, but not all downsets are polyhedral.
Our generalization relies on a natural extension of the definition of $k$-aggregation closure to downsets of $\R^n_+$. For any downset $D$ of $\R^n_+$, we denote by 
$$
\Lambda(D): = \{f \in \R^n \mid \sup\{f^\transp x \mid x \in D\} < \infty\}.
$$ 
In particular, note that $f^\transp x \leq \beta$ is valid for $D$ if and only if $f \in \Lambda(D)$ and $\beta \geq  \sup\{f^\transp x \mid x \in D\}$. Then, the \emph{$k$-aggregation closure} of $D$ is defined by
$$
\tA_k (D): = \bigcap_{f^1, \ldots, f^k \in \Lambda(D)} \conv(\{x \in \N^n \mid (f^j)^\transp x \leq \sup\{(f^j)^\transp d \mid d \in D\}, \ \forall j\in[k] \}).
$$

The next observation shows that $\tA_k$ is indeed a generalization of $\A_k$.

\begin{observation}\label{obs reduce}
For any packing polyhedron $P$ and any $k \ge 1$, we have $\tilde{\mathcal A}_k(P) = \A_k(P)$.
\end{observation}

\begin{proof}
Let $P = \{x \in \R^n_+ \mid Ax \leq b, \ x \geq 0\}$ be a packing polyhedron in $\R^n$.
It is simple to show that $\tA_k (P) \subseteq \A_k (P)$.
To see this, consider an inequality $\lambda^\transp A x \leq \lambda^\transp b$, for $\lambda \in \R^m_+$, in the definition of $\A_k (P)$.
Then $\lambda^\transp A x \leq \lambda^\transp b$ is valid for $P$.
Thus $\lambda^\transp A \in \Lambda(D)$, and $\sup\{\lambda^\transp A d \mid d \in P\} \leq \lambda^\transp b$.
Hence, the inequality $\lambda^\transp A x \leq \sup\{\lambda^\transp A d \mid d \in P\}$ in the definition of $\tA_k (P)$ implies the original inequality $\lambda^\transp A x \leq \lambda^\transp b$.

Next, we show $\tA_k (P) \supseteq \A_k (P)$.
Consider an inequality $f^\transp x \leq \sup\{f^\transp d \mid d \in P\}$, for $f \in \Lambda(P)$, in the definition of $\tA_k (P)$. This inequality is valid for $P$.
From Farkas' lemma we know that there exist some $\lambda \in \R^m_+$ and $\gamma \in \R^n_+$ such that $\lambda^\transp A - \gamma^\transp I = f^\transp$ and $\lambda^\transp b \leq \sup\{f^\transp d \mid d \in P\}$. 
Note that the inequality $\lambda^\transp A x \leq \lambda^\transp b$ is valid for $P$.
Furthermore, it dominates the inequality $f^\transp x \leq \sup\{f^\transp d \mid d \in P\}$ in the nonnegative orthant.
This is because,  whenever $x \geq 0$,
$$
f^\transp x = \lambda^\transp A x - \gamma^\transp I x \leq \lambda^\transp A x \leq \lambda^\transp b \leq \sup\{f^\transp d \mid d \in P\}.
$$
We have shown $\tA_k (P) \supseteq \A_k (P)$, which completes the proof of the observation.
\end{proof}

We now state our  generalization of Theorem~\ref{theo: k-aggregation} to downsets of $\R^n_+$.

\begin{theorem}
\label{theo: k-aggregation-general}
For any downset $D$ of $\R^n_+$ and any $k\geq 1$, the set $\tA_k(D)$ is a packing polyhedron.
\end{theorem}

In the special case $k=1$, the $k$-aggregation closure is also known as the \emph{aggregation closure}.
In the recent unpublished manuscript \cite{PasPoiPul19}, the authors independently show that the aggregation closure of a packing or covering rational polyhedron $P$ is polyhedral.
The main differences with our Theorem~\ref{theo: k-aggregation-general} are the following:
(i) The result in \cite{PasPoiPul19} holds for both packing and covering polyhedra, while our Theorem~\ref{theo: k-aggregation} only deals with the packing case;
(ii) The result in \cite{PasPoiPul19} requires the given set to be a polyhedron, while in our case the given set can be a general downset of $\R^n_+$;
(iii) The proof in \cite{PasPoiPul19} is direct, while our Theorem~\ref{theo: k-aggregation} is a consequence of Theorem~\ref{theo: packingset wqo};
(iv) In \cite{PasPoiPul19} the authors only discuss in detail the aggregation closure, and claim that an analogous proofs can be obtained for the $k$-aggregation closure, while in this paper we directly consider the $k$-aggregation closure.

\medskip

We can further extend the definition of $\tA_k$ to $k = \infty$ in a natural way:
$$
\tA_\infty (D): = \conv(\{x \in \N^n \mid f^\transp x \leq \sup\{f^\transp d \mid d \in D\}, \ \forall f \in \Lambda(D)  \}).
$$
We obtain the following result.

\begin{theorem}
\label{theo: LAST}
For any downset $D$ of $\R^n_+$, the set $\tA_\infty(D)$ is a packing polyhedron.
\end{theorem}

Our work sheds light onto the connection between Order Theory and polyhedrality of closures in Integer Optimization. Only few papers so far have explored this connection.
In \cite{averkov2012finitely}, Averkov exploits the Gordan-Dickson lemma to show the polyhedrality of the closure of a rational polyhedron obtained via disjunctive cuts from a family of lattice-free rational polyhedra with bounded max-facet-width. In the paper \cite{dash2017polyhedrality}, Dash et al.~consider fairly well-ordered qoset to extend the result of Averkov.
In particular, the authors prove the polyhedrality of the closure of a rational polyhedron with respect to any family of $t$-branch sets, where each set is the union of $t$ polyhedral sets that have bounded max-facet-width. 
Other recent polyhedrality results in Integer Optimization include \cite{AndLouWei10,dadush2011chvatal,dunkel2013gomory,dash2016polyhedrality,dash2019generalized}.

The organization of this paper is as follows: 
In Section~\ref{sec: preliminaries} we present some preliminaries and notations from Order Theory that will be used in our proofs.
In Section~\ref{sec: main proof} we show Theorem~\ref{theo: packingset wqo}, while in Section~\ref{sec: polyhedrality} we provide a proof of Theorem~\ref{theo: k-aggregation}.
In Section~\ref{sec extensions} we turn our attention to non-polyhedral sets and prove Theorem~\ref{theo: k-aggregation-general} and Theorem~\ref{theo: LAST}.

\section{Preliminaries in Order Theory}
\label{sec: preliminaries}

Recall that a \emph{quasi-order} is a binary relation $\order$ over a set $X$ that is reflexive and transitive. If $a\order b$, we also write $b\georder a$. If $a\order b$ or $b\order a$, the elements $a$ and $b$ are said to be \emph{comparable}. If both $a\order b$ and $b\order a$, then we write $a\sim b$ (which is an equivalence relation). A sequence $x_1,x_2,\ldots$ of elements from $X$ is said to be \emph{increasing} if $x_1\order x_2\order\dots$ and \emph{decreasing} if $x_1\georder x_2\georder \dots$.  

Most quasi-orders in this paper will in fact be \emph{partial orders}, that is, they are \emph{antisymmetric}: $a\order b$ and $b\order a$ imply $a=b$. In particular, we will consider the subset relation on $\R^n$ (and the induced partial order on integer packing sets), and the partial order on $(\N^n,\leq)$ given by the component-wise comparison: $x\leq y$ if and only if $x_i\leq y_i$ for all $i\in[n]$.

A quasi-order $(X,\order)$ is a \emph{well-quasi-order} (\emph{wqo}) if for any infinite sequence of elements $x_1, x_2, \ldots$ from $X$ there are indices $i<j$ such that $x_i \order x_j$. A quasi-order $(X, \order)$ is said to have the \emph{finite basis property} if for all $X' \subseteq X$, there exists a finite subset $B \subseteq X'$ such that for every $x\in X'$ there is a $b\in B$ such that $b\order x$. The next result provides us with characterizations of well-quasi-orders. 
\begin{lemma}[{\cite[Theorem 2.1]{higman1952ordering}}]\label{lem: wqo_eq1}
Let $(X, \order)$ be a quasi-order. The following statements are equivalent:
\begin{enumerate}
\item[\textup{(i)}] $(X, \order)$ is a wqo;
\item[\textup{(ii)}] $(X, \order)$ has the finite basis property;
\item[\textup{(iii)}] every infinite sequence of elements from $X$ has an infinite increasing subsequence.
\end{enumerate}
\end{lemma}

Given two quasi-orders $(X_1,\order_1)$ and $(X_2,\order_2)$, the \emph{product quasi-order} is $(X_1\times X_2, \order)$ where $(x_1,x_2)\order (y_1,y_2)$ if and only if $x_1\order_1 y_1$ and $x_2\order_2 y_2$.
\begin{lemma}
Let $(X_1,\order_1)$ and $(X_2,\order_2)$ be wqo's. Then the product quasi-order is a wqo.
\end{lemma}
The proof of this well-known fact follows easily from the equivalence of (i) and (iii) in Lemma~\ref{lem: wqo_eq1}: given an infinite sequence of elements in $X_1\times X_2$, we can find an infinite subsequence for which the components in $X_1$ form an increasing sequence, and then a further subsequence in which also the components in $X_2$ form an increasing sequence. The resulting subsequence is an increasing subsequence in $X_1\times X_2$.  

Since $(\N,\leq)$ is a wqo, the lemma implies that for any positive $n$ the set $\N^n$ is a wqo under the usual component-wise comparison. 
\begin{lemma}[Gordan-Dickson, \cite{dickson1913finiteness}]
\label{gdl: a}
The poset $(\N^n, \leq)$ is a wqo.
\end{lemma}

Given a quasi order $(X,\order)$ we denote by $X^*$ the set of all finite sequences of elements from $X$. We define a quasi order $\order^*$ on $X^*$ by setting $(x_1,\ldots, x_n)\order^*(y_1,\ldots, y_m)$ if and only if there is a strictly increasing function $f:[n]\to [m]$ such that $x_i\order y_{f(i)}$ for all $i\in [n]$. In this paper we will need the following generalization of the Gordon-Dickson lemma.
\begin{lemma}[Higman's lemma, \cite{higman1952ordering}]
Let $(X,\order)$ be a wqo. Then $(X^*,\order^*)$ is a wqo as well.
\end{lemma}

\section{Integer packing sets are well-quasi-ordered}
\label{sec: main proof}

In this section we prove our main result that integer packing sets in $\R^n$ form a wqo under inclusion. The proof is based on the following lemma.
\begin{lemma}\label{X**}
Let $(X,\order)$ be a wqo. Define $X^{**}$ to be the set of decreasing sequences in $X$: 
$$X^{**}=\{(x_0,x_1,\ldots)\in X^{\N}: x_0\georder x_1\georder\cdots\}.$$ 
For $x,y\in X^{**}$ set $x\order^{**}y$ if $x_i\order y_i$ for every $i\in \N$. Then $(X^{**}, \leq ^{**})$ is a wqo.
\end{lemma}
\begin{proof}
We start with the following claim.
\begin{claim} Let $x\in X^{**}$. There is a $k\in \N$ such that $x_k\sim x_\ell$ for all $\ell\geq k$.
\end{claim}
\begin{cpf}
Since $X$ is a wqo, it follows by Lemma~\ref{lem: wqo_eq1} that there is a finite subset $I\subseteq \N$ of indices such that for any $\ell\in \N$ there is an $i\in I$ with $x_\ell\georder x_i$. Let $k$ be the largest index in $I$. Consider any $\ell\geq k$. Since $x_0,x_1,\ldots$ is decreasing, we have $x_\ell\order	x_k$, but also $x_\ell\georder x_i\georder x_k$ for some $i\in I$. Hence, $x_\ell\sim x_k$.
\end{cpf}
We call the smallest $k$ as in the claim the \emph{tail} of $x$. By Higman's lemma, it follows that the product quasi-order $\order'$ on $X^*\times X$ is a wqo.  Let $\phi:X^{**}\to X^*\times X$ be defined by $\phi(x)=((x_0,\ldots, x_{k-1}),x_k)$, where $k$ is the tail of $x$. Let $x,y\in X^{**}$ and suppose that $\phi(x)\order' \phi(y)$. To complete the proof, it suffices to show that $x\order^{**}y$. 

Let $k$ and $l$ be the tails of $x$ and $y$, respectively. Since $\phi(x)\order' \phi(y)$ we have a strictly increasing function $f:\{0,\ldots, k-1\}\to \{0,\ldots, \ell-1\}$ such that $x_i\order y_{f(i)}$ for every $i\in \{0,\ldots, k-1\}$. 

Since $y_0\georder y_1\georder\cdots$ and $f(i)\geq i$ for every $i\in \{0,\ldots, k-1\}$, we have $x_i\order y_{f(i)}\order y_i$ for every $i\in\{0,\ldots, k-1\}$. Since $x_k\order y_\ell$ (and hence $x_i\order x_k\order y_\ell\order y_j$ for every $i\geq k$ and every $j\in\N$) we in fact have $x_i\order y_i$ for all $i\in \N$.   
\end{proof}

We will now prove Theorem~\ref{theo: packingset wqo}: the set of integer packing sets in $\R^n$ is a wqo under inclusion.
\begin{proof}[Proof of Theorem~\ref{theo: packingset wqo}] The proof is by induction on $n$. The case $n=1$ follows directly from the fact that $(\N,\leq)$ is a wqo. For the induction step, we associate to any integer packing set $S\subseteq \R^{n+1}$ a sequence $(S_0,S_1,\ldots)$ of `slices' by setting 
$$S_i=\{(x_1,\ldots, x_n)\in \N^n : (x_1,\ldots, x_n, i)\in S\}.$$ 
As $S$ is an integer packing set, it follows that the $S_i$ are integer packing sets in $\R^n$ and that $S_0\supseteq  S_1\supseteq \cdots$. For two packing sets $S,T$ in $\R^{n+1}$ we have $S\subseteq T$ if and only if for the corresponding slices we have $S_i\subseteq T_i$ for all $i\in \N$. Hence, the well-quasi-ordering of integer packing sets in $\R^{n+1}$ follows from that of integer packing sets in $\R^n$ by Lemma~\ref{X**}.
\end{proof}

As a consequence to Theorem~\ref{theo: packingset wqo} we obtain the following structural result about integer packing sets.
An \emph{$n$-dimensional block} is a set of the form $X_1\times\cdots\times X_n$, where each $X_i$ is equal to $\N$ or to $[n]$ for some $n\in \N$. 
\begin{corollary} Let $Q$ be an integer packing set in $\R^n$. Then $Q$ is the union of finitely many $n$-dimensional blocks. 
\end{corollary}
\begin{proof}
The proof is by induction on $n$. If $n=1$, then any integer packing set in $\R^n$ is an $n$-dimensional block. Now suppose that the statement holds for a given $n$ and consider an integer packing set $Q$ in $\R^{n+1}$. Define the $n$-dimensional slices $Q_i=\{(x_1,\ldots, x_n): (x_1,\ldots, x_n,i)\in Q\}$ for every $i\in\N$. Then $Q_0, Q_1,\ldots$ is a decreasing sequence of integer packing sets in $\R^n$. Hence, by Theorem~\ref{theo: packingset wqo}, there is a $k\in \N$ such that $Q_k=Q_\ell$ for all $\ell\geq k$. By assumption, each set $Q_i$ is a union of finitely many $n$-dimensional blocks. Hence $Q_i\times \{0,1,\ldots, i\}$ is a union of finitely many $n+1$-dimensional block for any $i=0,1,\ldots, k-1$, and also $Q_k\times \N$ is a union of finitely many $n+1$-dimensional blocks. Since 
$$Q=(Q_k\times \N)\cup \bigcup_{i=0}^{k-1} Q_i\times \{0,1,\ldots, i\},$$
the result follows.  
\end{proof}

\section{Polyhedrality of the $k$-aggregation closure}
\label{sec: polyhedrality}

In this section we prove that the $k$-aggregation closure of a packing polyhedron is itself a packing polyhedron (Theorem~\ref{theo: k-aggregation}). We will use some standard notation from polyhedral theory. In particular, given $A\subseteq \R^n$, we denote by $\conv(A)$ the convex hull of $A$, and given a polyhedron $P\subseteq \R^n$, we denote by $P_I=\conv(P\cap \Z^n)$ the integer hull of $P$. Given $a\in \R^n$, we define $a_+\in\R^n$ by $(a_+)_i:=\max\{0,a_i\}$ for all $i\in[n]$. 

\begin{lemma}\label{lem: domination}
Let $D$ be a downset of $\R_+^n$ and let $a^\transp x\leq \beta$ be a valid inequality for $D$. Then $a_+^\transp x\leq \beta$ is valid for $D$.
\end{lemma}
\begin{proof}
Let $x\in D$ and let $x'\in \R^n$ be defined by $x'_i=x_i$ if $a_i\geq 0$ and $x'_i=0$ if $a_i<0$. Since $D$ is a downset, we have $x'\in D$. Hence, $a_+^\transp x=a^\transp x'\leq \beta$.
\end{proof}

\begin{lemma}\label{lem: downset-in-R+}
A polyhedron is a downset of $\R^n_+$ if and only if it is a packing polyhedron.
\end{lemma}

\begin{proof}
It is simple to see that every packing polyhedron is a downset of $\R^n_+$. 
For the converse implication, let $P$ be a polyhedron that is a downset of $\R^n_+$. Then $P$ can be written in the form 
$$P=\{x\in \R^n\mid x\geq 0,\quad  (a^i)^\transp x\leq b_i,\quad i\in [m]\}.$$
Consider any inequality $(a^i)^\transp x\leq b_i$. Since $P$ is a downset, it follows from Lemma~\ref{lem: domination} that $(a^i_+)^\transp x\leq b_i$ is valid for $P$. Moreover, the inequality $(a^i)^\transp x\leq b_i$ is implied by the inequalities $(a^i_+)^\transp x\leq b_i$ and $x\geq 0$. 

It follows that 
$$P=\{x\in \R^n\mid x\geq 0,\quad (a^i_+)^\transp x\leq b_i,\quad i\in [m]\}.$$
Since $0\in P$, it follows that $b_i\geq 0$ for every $i\in[m]$. We conclude that $P$ is a packing polyhedron.
\end{proof}

\begin{lemma}\label{lem: integerhull} Let $P$ be a packing polyhedron. Then the integer hull $P_I$ is also a packing polyhedron.
\end{lemma}
\begin{proof}
We can write $P=\{x\in \R^n\mid x\geq 0,\ (a^i)^\transp x\leq b_i,\ i\in [m]\}$ where the $a^i$ and $b_i$ are nonnegative. Consider any of the inequalities $(a^i)^\transp x\leq b_i$. Note that since $a^i$ is nonnegative, the set $\{(a^i)^\transp x\mid x\in \N^n\}\cap (b_i, b_i+1)$ is finite. Hence, there exist nonnegative rational $c^i\leq a^i$ and $d_i\geq b_i$ such that 
$$
\{x\in \N^n\mid (a^i)^\transp x\leq b_i\}=\{x\in \N^n\mid (c^i)^\transp x\leq d_i\}.
$$ 
Let $P'=\{x\in \R^n\mid x\geq 0,\ (c^i)^\transp x\leq d_i,\ i\in [m]\}$. Then $P\cap \N^n=P'\cap \N^n$ and hence $P_I=P'_I$. By Meyer's theorem \cite{meyer1974existence}, the integer hull of a rational polyhedron is itself a polyhedron. Hence, $P_I=P'_I$ is a polyhedron.

It is clear that the polyhedron $P_I$ is a downset of $\R_+^n$. Hence, by Lemma~\ref{lem: downset-in-R+}, $P_I$ is a packing polyhedron.  
\end{proof}

Now we are ready to present the proof of Theorem~\ref{theo: k-aggregation}.

\begin{proof}[Proof of Theorem~\ref{theo: k-aggregation}]
Let $P$ be a packing polyhedron defined by $P = \{x \in \R^n \mid A x \leq b, \ x \geq 0\}$, and let $k$ be a positive integer. Denote by $\mathcal{P}$ the collection of polyhedra of the form 
$$\{x\in\R^n \mid x\geq 0,\quad (\lambda^j)^\transp A x \leq (\lambda^j)^\transp b,\quad \forall j \in [k]  \},$$
for all possible $\lambda^1, \ldots, \lambda^k \in \R^m_+$. 

Since $A$ is nonnegative, for every $Q\in \mathcal{P}$ the set $Q\cap \N^n$ is an integer packing set. By Theorem~\ref{theo: packingset wqo}, the set of integer packing sets in $\R^n$ is a wqo under inclusion. Hence, it follows from the finite basis property that there is a finite subset $\mathcal{P}'\subseteq \mathcal{P}$ such that for any $Q\in \mathcal{P}$ there is a $Q'\in\mathcal{P'}$ such that $Q'\cap\N^n\subseteq Q\cap \N^n$, and hence also that $Q'_I\subseteq Q_I$. We conclude that 
$$
\A_k(P)=\bigcap \big\{Q_I: Q\in \mathcal{P}\big\}=\bigcap \big\{Q_I: Q\in \mathcal{P'}\big\}.
$$
Since by Lemma~\ref{lem: integerhull} the integer hull $Q_I$ is a packing polyhedron for every $Q\in\mathcal{P}'$, and the intersection of finitely many packing polyhedra is again a packing polyhedron, it follows that $\A_k(P)$ is a packing polyhedron. 
\end{proof}

\section{Generalization to non-polyhedral sets}
\label{sec extensions}

In this section, we provide the proofs of our generalizations of Theorem~\ref{theo: k-aggregation} to non-polyhedral sets.
In particular, we give the proofs of Theorem~\ref{theo: k-aggregation-general} and of Theorem~\ref{theo: LAST}.
We refer the reader to Section~\ref{sec intro} for the definitions of $\tA_k$ and of $\tA_\infty$.

\begin{proof}[Proof of Theorem~\ref{theo: k-aggregation-general}]
Define $\Lambda^+(D): = \Lambda(D) \cap \R^n_+$. We first show that in the definition of $\tA_k(D)$ we can replace $\Lambda(D)$ with $\Lambda^+(D)$, i.e., 
$$
\tA_k(D) = \bigcap_{f^1, \ldots, f^k \in \Lambda^+(D)} \conv(\{x \in \N^n \mid (f^j)^\transp x \leq \sup\{(f^j)^\transp d \mid d \in D\}, \ \forall j \in [k] \}).
$$
The containment $\subseteq$ is trivial, thus we only need to show the containment $\supseteq$.
Let $f \in \Lambda(D)$, and consider the associated valid inequality for $D$ given by $f^\transp x \leq  \sup\{f^\transp d \mid d \in D\}$. Since $D$ is a downset of $\R^n_+$, we know from Lemma~\ref{lem: domination} that $(f^+)^\transp x \leq \sup\{f^\transp d \mid d \in D\}$ is also valid for $D$, and dominates the original inequality in $\R^n_+$.
In particular, this implies that $\sup\{(f^+)^\transp d \mid d \in D\} \leq \sup\{f^\transp d \mid d \in D\}$, hence $f^+\in \Lambda^+(D)$ since the latter supremum is finite by assumption.
Hence, we have shown that $(f^+)^\transp x \leq \sup\{(f^+)^\transp d \mid d \in D\}$ dominates the original inequality $f^\transp x \leq  \sup\{f^\transp d \mid d \in D\}$ in $\R^n_+$. 
We have therefore proven the containment $\supseteq$. 

Lastly, we follow almost the exact same argument as in the proof of Theorem~\ref{theo: k-aggregation}, except now we consider the collection $\mathcal{P}$ of polyhedra of the form 
$$\{x \in \R^n \mid x\geq 0,\ (f^j)^\transp x \leq \sup\{(f^j)^\transp d \mid d \in D\}, \ \forall j \in [k]\},$$ 
for all possible $f^1, \ldots, f^k \in \Lambda^+(D)$. 
\end{proof}

We now turn our attention to the set $\tA_\infty$.

\begin{proof}[Proof of Theorem~\ref{theo: LAST}]
For any $f\in \Lambda(D)$ we denote $\beta_f:=\sup\{f^\transp d \mid d \in D\}$. As in the previous proof, we can restrict to $f\in \Lambda^+(D):=\Lambda(D) \cap \R^n_+$ in the definition of $\tA_\infty(D)$: 
$$
\tA_\infty(D)=\conv(\{x \in \N^n \mid f^\transp x \leq \beta_f,\ \forall f \in \Lambda^+(D)  \}). 
$$
For any $f\in \Lambda^+(D)$ let $S_f=\{x\in \N^n\mid f^\transp x\leq \beta_f\}$. Then $S_f$ is an integer packing set in $\R^n$. By Theorem~\ref{theo: packingset wqo}, the set of integer packing sets in $\R^n$ is a wqo under inclusion. Hence, it follows from the finite basis property that there is a finite subset $B\subseteq \Lambda^+(D)$ such that for every $f\in \Lambda^+(D)$ there is a $f'\in B$ for which $S_{f'}\subseteq S_f$. It follows that
$$
\tA_\infty(D)=\conv(\{x \in \N^n \mid f^\transp x \leq \beta_f, \ \forall f \in B \}).
$$
By Lemma~\ref{lem: integerhull}, it follows that $\tA_\infty(D)$ is a packing polyhedron. 
\end{proof}

\bibliography{cite}

\end{document}